\documentclass{article}
\usepackage[utf8]{inputenc}
\usepackage[margin=1.5in]{geometry}
\usepackage{todonotes}
\usepackage{caption}
\usepackage{nicefrac}
\usepackage{amsmath,enumerate}
 \usepackage{amsthm}
\usepackage{amsfonts}
\usepackage{amssymb}
\usepackage{mathrsfs}
\usepackage{pgfplots}
\usepackage{mathtools}
\usepackage{comment}
\usepackage{lmodern}
\usepackage{subfig}
\usepackage{caption}
\usepackage[style= numeric-comp,hyperref=true, doi=false,url=false,
            isbn=false,
            firstinits=true, sorting = none, 
            block=none, backend=bibtex,maxnames=99]{biblatex}
\renewbibmacro{in:}{} 
\bibliography{strucave}

\usepackage{tikz}
\usepackage{tkz-fct}
\usepackage{tkz-euclide}
\usetikzlibrary{external}
\usetikzlibrary{shapes.geometric}
\usetikzlibrary{decorations.shapes,arrows,calc,decorations.markings,shapes,decorations.pathreplacing,shapes.arrows}
\usetikzlibrary{positioning}
\usetikzlibrary{topaths,calc}
\tikzset{main node/.style={circle,fill=blue!20,draw,inner sep=1pt},}
  
  \tikzset{highlight1/.style={rectangle,
                           fill=red!15,
                           rounded corners = 0.5 mm,
                           inner sep=1pt,
                           fit=#1}}

\newcommand\hlight[1]{\tikz[overlay, remember picture,baseline=-\the\dimexpr\fontdimen22\textfont2\relax]\node[rectangle,fill=blue!13,rounded corners,fill opacity = 0.2,draw,thick,text opacity =1] {$#1$};}
  
\usetikzlibrary{fit}

\newtheorem{Theorem}{Theorem}
\newtheorem{Definition}{Definition}
\newtheorem{Example}{Example}
\newtheorem{Proposition}[Theorem]{Proposition}
\newtheorem{Lemma}{Lemma}

\newcommand{\xc}[1]{\vspace{.1cm}
\noindent {\em #1} }

\newcommand{\cT}{\mathcal{T}}

\newcommand{\R}{\mathbb{R}}

\newcommand{\bF}{\mathbb{F}}

\newcommand{\N}{\mathbb{N}}

\newcommand{\C}{\mathrm{C}}
\newcommand{\rd}{\mathrm{d}}
\newcommand{\V}{\mathbb{V}}

\newcommand{\G}{\mathcal{G}}

\newcommand{\dep}{\operatorname{dep}}

\newcommand{\supp}{\operatorname{supp}}
\DeclareMathOperator*{\lcm}{lcm}

\title{Sparse Linear Ensemble Systems and Structural Averaged Controllability: Single-input Case\thanks{Corresponding author: X. Chen.}}
\begin{document}
\author{Xudong Chen\thanks{X. Chen is with the Department of Electrical and Systems Engineering, Washington University in St. Louis. Email: \texttt{cxudong@wustl.edu}.}
\quad and \quad 
Bahman Gharesifard\thanks{B.~Gharesifard is with the Department of Electrical and Computer Engineering, University of California, Los Angeles. Email: \texttt{gharesifard@ucla.edu}.}
}

\date{}
\maketitle

\begin{abstract}
We consider continuum ensembles of linear time-invariant control systems with single inputs. 
A sparsity pattern is said to be structurally averaged controllability if it admits an averaged controllable linear ensemble system. We provide a necessary and sufficient condition for a sparsity pattern to be structurally averaged controllable.  
\end{abstract}

\section{Introduction}
In this paper, we address the problem of structural averaged controllability introduced in~\cite{gharesifard2021structural}. Specifically, we consider continuum ensembles of linear time-invariant control systems with single inputs. The $(A,b)$ pairs of the individual systems are sparse, sharing a common sparsity pattern.   
The sparsity pattern is said to be structurally averaged controllable if it admits sparse linear ensemble system that are averaged controllable. 
Precise definitions will be given shortly in Subsection~\ref{ssec:problemformulation}.  

Ensemble control originated from quantum spin systems~\cite{brockett2000stochastic,li2006control}, and provides an alternative approach for controlling large-scale multi-agent systems, which is by nature  resilient and scalable~\cite{chen2021sparse}.  
A major technical challenge of ensemble control stems from the requirement that the control input be generated irrespective of parameters of the individual systems. Over the last score, there have been steady efforts and progress made in obtaining necessary and/or sufficient conditions for ensemble systems to be controllable or even path-controllable (see, e.g.,~\cite{helmke2014uniform,li2015ensemble,agrachev2016ensemble,chen2019structure,dirr2021uniform}). 
When the individual systems are themselves networked control systems as the scenario considered in this paper, characterizations of network topologies that can sustain ensemble controllability have been obtained in~\cite{chen2021sparse,chen2019controllability}.   

However, the motivation for us to study averaged controllability and the relevant structural properties stems from a recent {\em negative} result:  
It has been shown~\cite{chen2023controllability,danhane2023conditions} that linear ensemble systems are problematic in terms of ensemble controllability when the underlying parameterization spaces are multidimensional. For this type of ``pathological'' case, one has to relax the notion of ensemble controllability and seek for the controllable characteristics of the whole system. 
We formulate these characteristics as integrated outputs, where the integration is taken over the parameterization space. The choice of such formulation is rooted in the study of statistical properties of the ensemble system.

A simple, yet important integrated output is the state-average. For linear ensemble systems, a necessary and sufficient condition for averaged controllability has been obtained in~\cite{zuazua2014averaged}. Building upon this condition, we proceed one step further and address the problem of structural averaged controllability, within the scope of linear ensemble systems.     

In the earlier work~\cite{gharesifard2021structural}, we focused on the class of linear ensemble systems whose parameterization spares are unit closed intervals equipped with the Lebesgue measure.  
We have demonstrated that the condition of being structurally averaged controllable is strictly weaker than the one of being structurally controllable for finite-dimensional linear systems~\cite{lin1974structural}. We have also exhibited a special class of sparsity patterns that are structurally averaged controllable --- these are the ones that have a single control-node and a special state-node which has a self-loop and is the predecessor of all other state-nodes (the correspondence between sparsity patterns and graphs will be given in Subsection~\ref{ssec:problemformulation}).     

In this paper, we consider a much broader class of linear ensemble systems whose parameterization spaces are allowed to be arbitrary continuum spaces with probability measures. We extend significantly the results of~\cite{gharesifard2021structural}. Amongst others, we provide in Theorem~\ref{thm:main} a complete characterization of sparsity patterns, with single control-nodes, that are structurally averaged controllable.   

The remainder of the paper is organized as follows: At the end of this section, we present key notions and notations used throughout the paper. In Section~\ref{sec:mainresult}, we formulate precisely the problem and state the main result. Section~\ref{sec:proof} is dedicated to the analyses and the proof of the main result. This paper ends with conclusions.      

\vspace{.1cm}

\noindent
{\bf Notations.} We gather here key notions and notations. 

\xc{Graphs.} Let $G = (V, E)$ be a directed graph (or simply digraph), possibly with self-loops. The node and edge sets of $G$ are $V$ and $E$, respectively. 
We denote by $v_iv_j$ a directed edge from $v_i$ to $v_j$; we call $v_i$ an {\em in-neighbor} of $v_j$, and $v_j$ an {\em out-neighbor} of $v_i$. For a subgraph $G' = (V',E')$ of $G$, we let
$$
N_{\rm out}(G'):=\{v_j\in V \mid \mbox{there is a } v_i \in V' \mbox{ s.t. } v_iv_j \in E\}.
$$
We define $N_{\rm in}(G')$ in a similar way.

A {\em walk}~$\tau$ from $v_i$ to $v_j$ is a sequence of nodes $\tau = v_{i_1}\ldots v_{i_k}$,
with $v_{i_1} = v_i$ and $v_{i_k} = v_j$, such that each $v_{i_j} v_{i_{j+1}}$, for $j =
1,\ldots ,k-1$, is an edge of $G$. The {\em length} of the walk, denoted by $\ell(\tau)$, is the
number of edges contained in it. A walk is said to be {\em closed} if its starting and ending nodes are the same.  
A walk is a {\em path} if there is no
repetition of nodes in the sequence. A walk is a {\em cycle} if there
is no repetition of nodes except the repetition of starting and
ending nodes. Note that a self-loop at a node $v_i$ is a cycle of
length~1. Given two walks $\tau_1 = v_{i_1}\ldots v_{i_k}$ and $\tau_2 = v_{j_1}\ldots v_{j_\ell}$ such that the ending node $v_{i_k}$ of $\tau_1$ coincides with the starting node $v_{j_1}$ of $\tau_2$, we can obtain a new walk by concatenating $\tau_1$ and $\tau_2$:
$$\tau=\tau_1 \tau_2:= v_{i_1}\ldots v_{i_k} v_{j_2} \cdots v_{j_\ell}.$$
In case $\tau_2$ has only a single node $v_{j_1}$, we have $\tau = \tau_1$.  
It should be clear that $\ell(\tau) = \ell(\tau_1) + \ell(\tau_2)$.

A node $v_j$ is said to be a {\em successor} of $v_i$ if there exists a path in $G$ from $v_i$ to~$v_j$, and $v_i$ is said to be a {\em predecessor} of~$v_j$. 
The digraph $G$ is said to be {\em strongly connected} if for any two distinct nodes $v_p$ and $v_q$, $v_q$ is both a successor and a predecessor of $v_p$. A graph with only a single node is, by default,  strongly connected. Such a graph is said to be {\em trivial} if it has no self-loop.

A node $v_r$ is said to be a {\em root} of $G$ if any other node is a successor of $v_r$.     
The digraph $G$ is said to be a {\em directed tree} if it has a unique root $v_r$ and for any other node $v_i$ of $G$, there exists a unique path from $v_r$ to $v_i$.

Given a subset $V'$ of $V$, the subgraph $G' = (V', E')$ of $G$ is said to be {\em induced by} $V'$ if the edge set $E'$ satisfies the following condition: if two nodes $v_i,v_j\in V'$ are such that $v_iv_j\in E$, then $v_iv_j\in E'$.  

Let $G = (V, E)$ and $H = (W,F)$ be two digraphs. A {\em graph homomorphism} $\pi: G\to H$ is a map from $V$ to $W$ such that if $v_iv_j$ is an edge in $G$, then $\pi(v_i)\pi(v_j)$ is an edge in $H$. 
One can extend the map $\pi$ to the edge sets $\pi:E\to F$, sending $v_iv_j$ to $\pi(v_iv_j) := \pi(v_i)\pi(v_j)$. 
Then, for a subgraph $G' = (V',E')$ of $G$, we let $\pi(G')$ be the subgraph of $H$ with $\pi(V')$ the node set and $\pi(E')$ the edge set. Conversely, for a subgraph $H' = (W', F')$ of $H$, we define $\pi^{-1}(H'):= (\pi^{-1}(W'), \pi^{-1}(F'))$.   

\xc{Miscellanies.} Let $G = (V, E)$ be a graph on $n$ nodes, and $V' = \{v_{i_1},\cdots,v_{i_k}\}$ be a subset of $V$, with $i_1 < \cdots < i_k$. 
For a vector $x\in \R^n$, we let $x |_{V'} \in \R^k$ be the subvector of $x$ given by
$x |_{V'} := (x_{i_1},\ldots, x_{i_k})$. The notation extends to matrices: specifically, for $X\in \R^{n\times m}$ with $x_1,\ldots, x_m\in \R^n$ its columns, we let 
$$  
X |_{V'} := \left [ x_1 |_{V'},\cdots, x_n |_{V'}\right ].
$$

For a matrix $C\in \R^{n\times m}$, we let $\|C\|$ be the induced matrix 2-norm.     
Let $\Sigma$ be an arbitrary topological space. A function $f: \Sigma\to \R^{n\times m}$ is {\em bounded} if exists a number $\gamma>0$ such that $\|f(\sigma)\| < \gamma$ for all $\sigma\in \Sigma$. We denote by $\C^0_b(\Sigma, \R^{n\times m})$ the space of all bounded, continuous functions from $\Sigma$ to $\R^{n\times m}$. For each $f\in \C^0_b(\Sigma, \R^{n\times m})$, we define
$$
\|f\|_\infty := \sup_{\sigma\in\Sigma} \|f(\sigma)\|.
$$
The topology on $\C^0_b(\Sigma, \R^{n\times m})$ induced by the norm is called the {\em uniform topology}. 

In this paper, we will deal with matrices with finitely many row, but infinitely many columns. Given such a matrix 
$$C = [C_1, C_2, C_3, \cdots ], \quad \mbox{with } C_j\in \R^n \mbox{ for all } j \geq 1,$$   
we let $C_{[m]} := [C_1,\ldots, C_m]$ be the finite-dimensional submatrix of $C$ obtained by keeping only the first $m$ columns. We say that $C$ has full rank (i.e., rank~$n$) if there exists a positive integer $m$ such that $C_{[m]}$ has rank~$n$. 

For a positive integer $n$ and for two real numbers $p$ and $q$, we write $p \equiv_n q$ if $n$ is a divisor of $(p-q)$.

\section{Problem Formulation and Main Result}\label{sec:mainresult}
\subsection{Problem formulation}\label{ssec:problemformulation}
Let $\Sigma$ be a manifold, possibly with boundary, and $\mu$ be a Borel probability measure on $\Sigma$ whose support contains an open set. 
We consider a continuum ensemble of linear time-invariant systems driven by a single control input: 
\begin{equation}\label{eq:systemeq}
\frac{\partial}{\partial t} x(t,\sigma) = A(\sigma)x(t,\sigma) + b(\sigma) u(t), \quad \mbox{for }\sigma\in \Sigma,
\end{equation}
where $x(t,\sigma)\in \R^n$ is the state of the individual system indexed by~$\sigma$ at time~$t$, $u(t)\in \R$ is the common control input, and $A:\Sigma \to \R^{n\times n}$ and $b:\Sigma\to \R^n$ are bounded, continuous functions. The control input is said to be {\bf admissible} if for any $T>0$, the function $u:[0,T]\to\R$ is integrable. 

Let $\chi(t):\Sigma \to \R^n$ be the {\bf profile} of system~\eqref{eq:systemeq} at time~$t$, defined by sending a parameter $\sigma$ to the corresponding state $x(t,\sigma)$. It should be clear that if $\chi(0)$ is bounded and continuous, then so is $\chi(t)$ for any admissible control input. Denote by $\bar x(t)$ the average of the individual states at time~$t$: 
$$
\bar \chi(t): = \int_\Sigma \chi(t) \rd \mu = \int_\Sigma x(t,\sigma) \rd \mu.
$$
We have the following definition: 

\begin{Definition}\label{def:averagectrl}
System~\eqref{eq:systemeq} is {\bf averaged controllable} if for any initial profile $\chi(0)\in \mathrm{C}_b^0(\Sigma, \R^n)$, any target average $x^*\in \R^n$, and any time~$T>0$, 
there exists an admissible control input $u(t)$ such that the solution of~\eqref{eq:systemeq} generated by $u(t)$ satisfies $\bar \chi(T) = x^*$.
\end{Definition}

Since the ensemble system~\eqref{eq:systemeq} is determined by the $(A, b)$ pair, we will simply say that $(A, b)$ is averaged controllable if~\eqref{eq:systemeq} is. We have the following  necessary and sufficient condition adapted from~\cite{zuazua2014averaged} for averaged controllability:

\begin{Lemma}\label{lem:necsufforavectrl}
    Ensemble system~\eqref{eq:systemeq} is averaged controllable if and only if the following column-infinite matrix:
    \begin{equation}\label{eq:defcab}
    C(A, b) := \left [ \int_\Sigma b \, \rd\mu, \int_\Sigma A b \, \rd\mu, \int_\Sigma A^2 b \, \rd\mu, \cdots\right ]  
    \end{equation}
    is of full rank, i.e., rank~$n$.
\end{Lemma}

In this paper, we deal with sparse $(A, b)$ pairs. By convention, the sparsity pattern is represented by a digraph $G = (V, E)$ on $(n + 1)$ nodes as follows: 
\begin{itemize}
\item The node set $V$ is a disjoint union of two subsets: $V_\alpha := \{\alpha_1,\ldots, \alpha_n\}$ and a singleton $V_\beta := \{\beta\}$. The $\alpha$-nodes correspond to the $n$ scalar states $x_i$, and the $\beta$-node corresponds to the single input $u$.  
\item There is a directed edge from $\alpha_j$ to $\alpha_i$ if $a_{ij}\neq 0$. There is a directed edge from $\beta$ to $\alpha_i$ if $b_i\neq 0$.  The $\beta$-node does not have any in-neighbor.    
\end{itemize}
We call $G$, defined above, the digraph {\bf induced by} the pair $(A, b)$.

Let $\mathcal{G}_{n,1}$ be the collection of all weakly connected digraphs $G$ on $(n + 1)$ nodes, with $n$ $\alpha$-nodes and a single $\beta$-node, satisfying the condition that $\beta$ has no in-neighbor. 
Given a digraph $G\in \G_{n,1}$,  a pair $(A,b)\in \C_b^0(\Sigma, \R^{n\times (n+1)})$ is said to be {\bf compliant with} $G$ if the digraph induced by $(A, b)$ is a subgraph of $G$.  
Further, we let
\begin{equation*}
\V(G):= \{(A, b)\in \C_b^0(\Sigma, \R^{n\times (n+1)})  \mid (A, b) \mbox{ is compliant with } G\}.
\end{equation*}
It should be clear that $\V(G)$ is a subspace. 
We now have the following definition:

\begin{Definition}\label{def:strucavectrl}
A digraph $G\in \mathcal{G}_{n,1}$ is {\bf structurally averaged controllable} if there exists a pair $(A, b)\in \V(G)$ such that it is averaged controllable.  
\end{Definition}

\subsection{Main result}\label{ssec:mainresult}
In this subsection, we present a necessary and sufficient condition for $G\in \G_{n,1}$ to be structurally averaged controllable. Our presentation relies on the notion of the so-called strong component decomposition:

\begin{Definition}\label{def:scd}
Let $G = (V, E)$ be an arbitrary weakly connected digraph. The {\bf strong component decomposition (SCD)} of $G$ is a node-set decomposition $V = \cup_{i = 0}^N V_i$, where the $V_i$'s are disjoint, such that the following hold:
\begin{enumerate}
\item For each $i = 0,\ldots, N$, the subgraph $G_i$ induced by $V_i$ is strongly connected. 
\item Let $G'$ be an arbitrary strongly connected subgraph of $G$. Then, $G'$ is a subgraph of $G_i$ for some $i = 0,\ldots, N$. 
\end{enumerate}
We call $G_i$'s the {\bf strong components} of $G$. 
\end{Definition}

Note that a subgraph $G_i$ can be a single node (with or without self-loop). 
The $\beta$-node itself forms a strong component, which we denote by $G_0$. 

Let $G_{\rm cyc} = (V_{\rm cyc}, E_{\rm cyc})$ be the union of the nontrivial strong components (i.e., components that contain cycles).  
Said in another way, a strong component $G_i$ is {\em not} in $G_{\rm cyc}$ if and only if $G_i$ comprises a single node {\em without} a self-loop. 

Let $V^+_{\rm cyc}$ be the set of successors of nodes in $V_{\rm cyc}$, and $G^+_{\rm cyc}$ be the subgraph induced by $V^+_{\rm cyc}$. 
The essential part of $G$ that determines whether $G$ is structurally averaged controllable is introduced in the following definition:

\begin{Definition}\label{def:core}
Given a digraph $G \in \G_{n,1}$, let $V^*:= V - V^+_{\rm cyc}$ and $G^*$ be the subgraph of~$G$ induced by $V^*$. We call $G^*$ the {\bf core} of $G$. 
\end{Definition}

It should be clear from the construction that a node $v_i$ is in $G^*$ if and only if it is {\em not} a successor of any cycle in $G$. In particular, $G^*$ is an acyclic graph with $\beta$ the root. 

We illustrate the above notions of $G$ in Figure~\ref{fig:ex}. 

\begin{figure*}[htb!]
\begin{center}
\begin{tikzpicture}[scale=1.0,node distance=1.4cm]   
\node[main node] (1)  {$\alpha_1$};  
\node[main node] (b) [below of= 1] {$\beta$};
    \node[main node] (2) [left of=1]  {$\alpha_2$};    
    \node[main node] (5) [above 
    left  of=2]  {$\alpha_4$};
    \node[main node] (6) [below left of= 5] {$\alpha_6$};    
    \node[main node] (7) [left of= 6] {$\alpha_8$};
     \node[main node] (3) [right of= 1] {$\alpha_3$};
    \node[main node] (4) [above right of= 3] {$\alpha_5$};    
    \node[main node] (8) [below right of= 4] {$\alpha_7$}; 
    \node[main node] (9) [right of= 8] {$\alpha_9$};     
        
    \path[draw,shorten >=2pt,shorten <=2pt]    
    (1) edge [-latex] (3)    
    (6) edge [-latex] (2)
    (6) edge [-latex] (7)
    (4) edge [-latex] (8)
    (5) edge [-latex] (6)    
    (2) edge [-latex] (5)
    (b) edge [-latex] (1)
    (b) edge [-latex] (2)    
    (3) edge[bend left = 15, -latex] (4)    
    (4) edge[bend left = 15, -latex] (3)  
    (8) edge[loop below,min distance=6mm,-latex] (8)
    (8) edge[-latex] (3)
    (8) edge[-latex] (9);        
    \end{tikzpicture}
\end{center}
\caption{The digraph shown in the figure is weakly connected. The {\em SCD} yields~6 strong components, including~4 trivial ones $G_0 = \{\beta\}$, $G_1 = \{\alpha_1\}$, $G_4 = \{\alpha_8\}$, $G_5 = \{\alpha_9\}$, and 2 nontrivial ones, namely, $G_2$ the subgraph induced by $\{\alpha_2,\alpha_4,\alpha_6\}$ and $G_3$ the subgraph induced by $\{\alpha_3,\alpha_5,\alpha_7\}$.   
The subgraph $G_{\rm cyc}$ consists of $G_2$ and $G_3$. 
The subgraph $G^+_{\rm cyc}$ consists of $G_{\rm cyc}$, the nodes $\alpha_8$ and $ \alpha_9$, and the edges $\alpha_6\alpha_8$ and $\alpha_7\alpha_9$. Finally, the core $G^*$ is the subgraph induced by $ \{\beta,\alpha_1\}$.}\label{fig:ex}
\end{figure*}
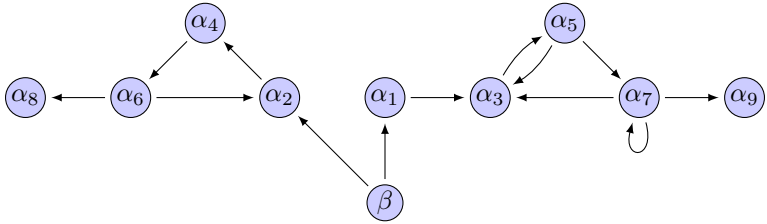

We now state the main result of this paper:
  
\begin{Theorem}\label{thm:main} 
Let $G\in \G_{n,1}$ and $G^*$ be its core. Then, $G$ is structurally averaged controllable if and only if $G^*$ contains a directed spanning  path.
\end{Theorem}

We equip with $\C^0_b(\Sigma, \R^{n\times (n+1)})$ the uniform topology, and $\V(G)$ the subspace topology. For $G\in \G_{n,1}$, let 
$$\V_*(G):= \{(A, b)\in \V(G) \mid (A, b) \mbox{ is averaged controllable} \}.$$
We have the following result:

\begin{Proposition}
    Suppose that $G$ is structurally averaged controllable; then, $\V_*(G)$ is open and dense in $\V(G)$. 
\end{Proposition}

\begin{proof}
We establish openness and density of $\V_*(G)$ subsequently. 

\xc{Proof that $\V_*(G)$ is open.} Let $(A,b)\in \V_*(G)$. Then, by Lemma~\ref{lem:necsufforavectrl}, the matrix $C(A, b)$ has full rank. Let~$m\geq n-1$ be such that the submatrix $C_{[m]}(A, b)$ has rank~$n$. Now, consider the map $\rho: \V(G)\to \R^{n\times m}$ defined by
\begin{equation}\label{eq:defrho}
\rho(A',b') := C_{[m]}(A',b'). 
\end{equation}
It should be clear that the map $\rho$ is continuous. Since $C_{[m]}(A,b)$ is of full rank, there exists an open neighborhood $\mathcal{V}$ of $C_{[m]}(A, b)$ in $\R^{n\times m}$ such that any matrix in $\mathcal{V}$ has rank~$n$. 
Then, $\mathcal{U}:= \rho^{-1}(\mathcal{V})$ is an open neighborhood of $(A,b)$ in $\V(G)$, contained in $\V_*(G)$.     

\xc{Proof that $\V_*(G)$ is dense.} The arguments below will be similar to those in~\cite{lin1974structural}. Since $G$ is structurally averaged controllable, there exists at least a pair $(A_*, b_*)$ in $\V_*(G)$. We still let $m\geq n-1$ be such that $C_{[m]}(A_*,b_*)$ has rank~$n$. For ease of presentation, we assume that $m = n - 1$ so that $C_{[m]}(A_*,b_*)$ is a square matrix of full rank (otherwise, one can always pick $n$ columns out of $C_{[m]}(A_*,b_*$) to obtain such a matrix).  
Let $\rho$ be defined as in~\eqref{eq:defrho}.  
Now for any $(A, b)\in \V(G)$, we consider a polynomial map $\delta: \R \to \R$ defined as follows:  
$$ \delta(s) := \det(\rho(A'(s),b'(s))), 
$$
where
$$(A'(s), b'(s)) := s(A_*,b_*)+ (1-s)(A, b).$$  
Since $(A'(1),b'(1))= (A^*,b^*)$, we have that $\delta(1)\neq 0$. Thus, $\delta$ is not identically zero, so it has at most $n$ distinct real roots.   
This, in particular, implies that any open neighborhood of~$0$ in $\R$ contains some $s$ such that $\delta(s)\neq 0$.  Since $(A'(0),b'(0)) = (A, b)$, we have that there exist an arbitrarily small $s > 0$ such that $$(A,b) + s(A_*,b_*)\in \V_*(G).$$
This completes the proof.  
\end{proof}

\section{Analyses and Proofs}\label{sec:proof}
This section is dedicated to the proof of Theorem~\ref{thm:main}. In Subsection~\ref{ssec:necessary}, we establish the necessity: We show that if $G$ is structurally averaged controllable, then $G$ satisfies the condition in Theorem~\ref{thm:main} (i.e., the condition that the core of $G$ contains a spanning path). The proof of sufficiency is more involved, and we outline below a sketch of proof. 

\xc{Sketch of proof for sufficiency.} 
The property of being structurally averaged controllable is monotone with respect to edge addition. Specifically, if $G = (V, E)$ is structurally averaged controllable, then so is any graph $G' = (V, E')$ obtained by adding new edges into $G$ (i.e., $E' \supsetneq E$). 
In Subsection~\ref{ssec:reducedgraph}, we introduce a special class of graphs $G$, termed reduced graphs, satisfying the following two properties: (1) $G$ satisfies the condition of Theorem~\ref{thm:main}, and (2) any graph $G'$ satisfies this condition can be translated, via edge removal, to a reduced graph. By monotonicity, to establish sufficiency, it suffices to show that every reduced graph is structurally averaged controllable. 

The reason we choose to work with reduced graph is for ease of analysis and computation: Specifically, for $G$ a reduced graph, we will be able to obtain an explicit  characterization of all walks with $\beta$ the starting node. Such a characterization is instrumental in analyzing and computing the infinite matrices $C(A, b)$, for $(A,b)\in \V(G)$. More specifically, the $ij$th entry of $C(A, b)$ is, in general, given by $\sum_\tau P(\tau)$, where the sum is over all walks $\tau$ of length~$j$ from $\beta$ to $\alpha_i$, and $P(\tau)$ is the product of the entries in $(A, b)$ corresponding to the edges in~$\tau$. If $G$ is reduced, then such a walk $\tau$, if exists, is unique.      

In Subsection~\ref{ssec:walksandreachablesets}, we characterize the walks in reduced graphs and describe reachable sets of nodes in $G$ (i.e., nodes that can be reached from $\beta$ by walks of particular length). Then, in Subsection~\ref{ssec:constructionAb}, we construct a particular $(A,b)$ pair in $\V(G)$ for $G$ reduced. Each nonzero entry of $(A, b)$ will be of the form $f^{\nu(e)}$, where $e$ is the edge of $G$ corresponding to the entry, $f: \Sigma \to \R_{\geq 0}$ is some continuous function, and $\nu: E\to \R_{\geq 0}$ specifies the power of~$f$ for each edge~$e$. 
In this way, the $ij$th entry of $C(A, b)$, if nonzero, can be expressed as 
\begin{equation}\label{eq:expressionofcij}
c_{ij} = f^{\sum_{e\in \tau}\nu(e)},
\end{equation}
where $\tau$ is the unique walk from $\beta$ to $\alpha_i$ of length~$j$. The expression~\eqref{eq:expressionofcij} will be made explicit in Subsection~\ref{ssec:computationcab}.  

Leveraging these computational results, we focus in Subsection~\ref{ssec:submatricescab} on a class of submatrices of $C(A,b)$ and show that they are full rank. Finally, in Subsection~\ref{ssec:cabavecontrol}, we show that the entire matrix $C(A, b)$ has rank~$n$.  \hfill{\qed}

\subsection{Proof of Necessity}\label{ssec:necessary}
In this subsection, we establish the following result: 

\begin{Proposition}
    Let $G\in \G_{n,1}$ and $G^*$ be its core. If $G^*$ does not have a directed spanning  path, then $G$ is not structurally averaged controllable.    
\end{Proposition}

\begin{proof}
Let $n^*:= |V^*|$ be the order of $G^*$. By relabelling (if necessary) the nodes of $G$, we can assume that the first $n^*$ nodes $\alpha_1,\ldots, \alpha_{n^*}$ are in $G^*$ and the remaining nodes are in $G^+_{\rm cyc}$. It follows that for any pair $(A, b)$ in $\V(G)$, the matrix $A$ can be written as follows:
$$
A = 
\begin{bmatrix}
A_{11} & 0 \\
A_{21} & A_{22} 
\end{bmatrix},
$$
where $A_{11}$ is $n^*$-by-$n^*$. This holds because if the block $A_{12}$ is nonzero, then there exists at least one node in $G^*$ such that it is a successor of some node in $G^+_{\rm cyc}$, which is a contradiction. Partitioning $x = (x_1; x_2)$ and $b = (b_1;b_2)$ correspondingly, we obtain 
$$
\begin{bmatrix}
\dot x_1(t,\sigma) \\
\dot x_2(t,\sigma)
\end{bmatrix} = 
\begin{bmatrix}
A_{11}(\sigma) & 0 \\
A_{21}(\sigma) & A_{22}(\sigma) 
\end{bmatrix}
\begin{bmatrix}
 x_1(t,\sigma) \\
 x_2(t,\sigma)
\end{bmatrix} + 
\begin{bmatrix}
b_1(\sigma) \\
b_2(\sigma)
\end{bmatrix} u(t).
$$
Note, in particular, that the dynamics of $x_1(t,\sigma)$ do not depend on $x_2(t,\sigma)$. It follows that if the pair $(A, b)$ is averaged controllable, then so is $(A_{11}, b_1)$. 
Consequently, if $G$ is structurally averaged controllable, then so is $G^*$. 

Thus, to establish the proposition, it suffices to show that if $G^*$ does not contain a directed spanning   path, then $G^*$ is not structurally averaged controllable. 
More specifically, we show below that under such a hypothesis of $G^*$, any pair $(A_{11}, b_1)\in \V(G^*)$ is not averaged controllable.

To the end, for a node $\alpha_i$ in $G^*$, let $\dep(\alpha_i)$ be the length of the {\em longest path} in $G^*$ from $\beta$ to $\alpha_i$. The depth of $\beta$ is set to be~$0$. 
Since $G^*$ does not contain a directed spanning path, 
\begin{equation}\label{eq:depthinGstar}
\dep(\alpha_i) < n^*, \quad  \mbox{for all } \alpha_i \in V^*.
\end{equation} 
Now, consider the entries of~$A^k_{11} b_1$. The $i$th entry of $A^k_{11}b_1$ is nonzero {\em only if} there exists a walk in $G^*$ of length~$(k+1)$ from $\beta$ to $\alpha_i$.   
Because $G^*$ is acyclic, all walks in $G^*$ are paths. It then follows from~\eqref{eq:depthinGstar} that 
$$A^{k}_{11} b_1 = 0, \quad \mbox{for all } k \geq n^*-1,$$
which implies that the matrix $C(A_{11},b_1)$ has at most $(n^* - 1)$ nonzero columns (i.e., the first $(n^*-1)$ columns) and, hence, its rank is less than $n^*$. By Lemma~\ref{lem:necsufforavectrl}, $(A_{11},b_1)$ is not averaged controllable. 
\end{proof}

\subsection{Reduced graphs}\label{ssec:reducedgraph}
In this subsection, we introduce a special class of digraphs, termed {\em reduced graphs} (Definition~\ref{def:reduced}), owning the property that any digraph satisfying the condition in Theorem~\ref{thm:main} can be reduced, via edge removal, to a reduced one.   

We start by introducing the notion of a skeleton graph $S$ associated with an arbitrary digraph $G$, which is  obtained by condensing the strong components of $G$ into single nodes. Precisely, we have 

\begin{Definition}\label{def:skeletongraph}
Let $G = (V, E)$ be weakly connected and $G_i$, for $i = 0,\ldots, N$, be the strong components of $G$. The {\bf skeleton graph} of $G$, denoted by $S = (W, F)$, is a directed graph on $(N+1)$ nodes $w_0,\ldots, w_N$ whose edge set $F$ is determined by the following rule: There exists a directed edge $w_i w_j$ if there exists a directed edge in $G$ from a node of $G_i$ to a node of $G_j$.   
\end{Definition}

Note that $S$ may have self-loops: A node $w_i$ has a self-loop if and only if $G_i$ has a cycle. 
By the second item of Definition~\ref{def:scd}, the skeleton graph $S$ will be acyclic if all its self-loops are removed.

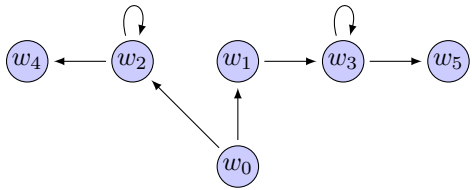
\begin{figure}[htb!]
\begin{center}
\begin{tikzpicture}[scale=1.0,node distance=1.4cm]
    \node[main node] (2) [below right of=1]  {$w_2$};
    \node[main node] (3) [left of= 2] {$w_4$};     
    \node[main node] (6) [right of=2]  {$w_1$};   
    \node[main node] (4) [below of= 6] {$w_0$};
    \node[main node] (7) [right of= 6] {$w_3$};       
    \node[main node] (9) [right of= 7] {$w_5$};    
         
    \path[draw,shorten >=2pt,shorten <=2pt]    
    (2) edge[loop above,min distance=6mm,-latex] (2)   
    (2) edge[-latex] (3)
    (4) edge[-latex] (6)
    (4) edge[-latex] (2)
    (6) edge[-latex] (7)
    (7) edge[-latex] (9)
    (7) edge[loop above,min distance=6mm,-latex] (7);       
    \end{tikzpicture}
\end{center}
\caption{The skeleton graph $S$ of the digraph shown in Figure~\ref{fig:ex}. Node $w_0$ corresponds to the $\beta$ node, $w_2$ and $w_3$ correspond to the two {\em nontrivial} strong components while the others correspond to the trivial ones.}  \label{fig:ex-sk}
\end{figure}
Next, we introduce the map $\pi: G \to S$, defined by sending a node $v_{i'}\in V_i$ to the node~$w_i$. It should be clear from Definition~\ref{def:skeletongraph} that the map $\pi$ is a graph homomorphism.

Let $G^*$ be the core of $G$ and $S^*:= \pi(G^*)$. It is not hard to see from the Definitions~\ref{def:core} and~\ref{def:skeletongraph} that $S^*$ is the core of $S$ and that $G^*$ and $S^*$ are isomorphic under $\pi |_{G^*}$. 
We now have the following definition:

\begin{Definition}\label{def:reduced}
A digraph $G \in \G_{n,1}$ is {\bf reduced} if it satisfies the following conditions: 
\begin{enumerate}
\item Let $S$ be the skeleton graph of $G$ and $S^*$ be the core of $S$. Then, $S$ is a directed tree and $S^*$ is a directed path. 
\item For each edge $w_iw_j$ of $S$, with $w_i \neq w_j$, the set $\pi^{-1}(w_iw_j)$ is a singleton.  
\item If $w_p\in S$ has a self-loop, then it is an out-neighbor of $S^*$. Moreover, for any such node $w_p$, $G_p = \pi^{-1}(w_p)$ is a cycle.  
\end{enumerate}
\end{Definition}

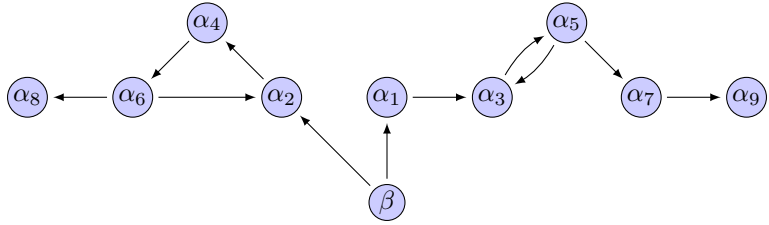
\begin{figure*}[htb!]
\begin{center}
\begin{tikzpicture}[scale=1.0,node distance=1.4cm]   
\node[main node] (1)  {$\alpha_1$};  
\node[main node] (b) [below of= 1] {$\beta$};
    \node[main node] (2) [left of=1]  {$\alpha_2$};    
    \node[main node] (5) [above 
    left  of=2]  {$\alpha_4$};
    \node[main node] (6) [below left of= 5] {$\alpha_6$};    
    \node[main node] (7) [left of= 6]  {$\alpha_8$};

    \node[main node] (3) [right of= 1] {$\alpha_3$};
    \node[main node] (4) [above right of= 3] {$\alpha_5$};    
    \node[main node] (8) [below right of= 4] {$\alpha_7$}; 
    \node[main node] (9) [right of= 8] {$\alpha_9$}; 
    \path[draw,shorten >=2pt,shorten <=2pt]    
    (1) edge[-latex] (3)    
    (6) edge[-latex]  (2)
    (6) edge[-latex]  (7)
    (4) edge[-latex]  (8)
    (5) edge[-latex] (6)    
    (2) edge[-latex] (5)
    (8) edge[-latex] (9)
    (b) edge[-latex] (1)
    (b) edge[-latex] (2)    
    (3) edge[bend left=15,-latex] (4)    
    (4) edge[bend left=15,-latex] (3);     
    \end{tikzpicture}
\end{center}
\caption{A reduced graph obtained by removing edges out of the digraph shown in Figure~\ref{fig:ex}.} \label{fig:ex-reduced}
\end{figure*}

For example, a reduced graph of the digraph shown in Figure~\ref{fig:ex} can be obtained by removing the edge $\alpha_7\alpha_3$ and the self-loop on $\alpha_7$, as shown in Figure~\ref{fig:ex-reduced}. Note that there are multiple ways to obtained a reduced graph. To wit, one can obtain another reduced graph by removing the edge $\alpha_5\alpha_3$ and the self-loop on $\alpha_7$ (while keeping the edge $\alpha_7\alpha_3$).

It should be clear that any reduced graph $G$ satisfies the condition in Theorem~\ref{thm:main}. We establish the following result:

\begin{Proposition}\label{prop:reducedgraph}
Let $G = (V, E)\in \G_{n,1}$ be such that its core $G^*$ contains a spanning path. Then, there is a subgraph $G' = (V, E')$ of $G$, with the same node set~$V$, such that $G'$ is reduced.   
\end{Proposition}

\begin{proof}
Let $S$ be the skeleton graph of $G$ and $\pi: G\to S$ be the graph homomorphism. 
Let $w_0:= \pi(\beta)$ and $S^*$ be the core of $S$. 
Since $G^*$ and $S^*$ are isomorphic, $S^*$ contains a spanning path, denoted by $P$.

Let $S^{(0)}$ be a directed spanning tree of $S$ such that it contains $P$ and all the self-loops. It should be clear that $P$ is the core of $S^{(0)}$. We will now remove edges out of $G$ leading to a subgraph $G^{(0)}$ whose skeleton graph is $S^{(0)}$.     
Consider the edges $w_iw_j$ in $S$, with $w_i \neq w_j$. There are two cases: ({\em i}) If $w_iw_j$ is {\em not} in $S^{(0)}$, then we remove all the edges in $\pi^{-1}(w_iw_j)$. ({\em ii}) If $w_iw_j$ is in $S^{(0)}$, then we remove all but one of the edges in $\pi^{-1}(w_iw_j)$.  
It should be clear that the resulting graph $G^{(0)}$ has $S^{(0)}$ as its skeleton graph and, moreover,  satisfies the first two items of Definition~\ref{def:reduced}. 

Next, we consider the nodes in $S^{(0)}$ with self-loops. For convenience, we label these nodes as $w_1,\ldots, w_q$, and let $G_1,\ldots, G_q$ be the corresponding strong components.  
We will now construct a reduced graph $G'$ by removing from each $G_p$, for $p = 1,\ldots, q$, some selected edges, which will be specified below.  
We will use $G^{(p)}$ to denote the graph obtained by removing the selected edges in $\cup_{i = 1}^p G_i$, $S^{(p)}$ the skeleton graph of $G^{(p)}$, and $\pi^{(p)}: G^{(p)}\to S^{(p)}$ the graph isomorphism.    
Along the edge-removal procedure, we will show by induction 
that $G^{(p)}$ satisfies the first two items of Definition~\ref{def:reduced} and, moreover, the core of $S^{(p)}$ is (isomorphic with) the path~$P$. 

The base case $p = 0$ has been established above. For the inductive step, we assume that $G^{(p-1)}$ has been obtained, with $1 \leq p \leq q$, and will construct $G^{(p)}$ by removing edges from $G_p$. Note that $G_p$ is a strong component of $G^{(p-1)}$. We will still use $w_p$ to denote the node $\pi^{(p-1)}(G_p)$ in $S^{(p-1)}$. 

By the induction hypothesis, $S^{(p-1)}$ is a directed tree and since $w_p$ cannot be the root, there is a unique node $w_i$ in $S^{(p-1)}$, with $w_i\neq w_p$, such that $w_i$ is an in-neighbor of $w_p$.  
Furthermore, there is a unique edge in $G^{(p-1)}$, denoted by $v_i\alpha_{p_0}$,  such that $\pi(v_i\alpha_{p_0}) = w_i w_p$. 
In case $w_p$ has out-neighbors other than itself, we denote them by $w_{j_1},\ldots, w_{j_m}$. Similarly, for each $\ell =1, \ldots, m$, there is a unique edge $\alpha_{p_\ell} \alpha_{j_\ell}$ such that $\pi(\alpha_{p_\ell}\alpha_{j_\ell}) = w_pw_{j_\ell}$. It should be clear that the nodes $\alpha_{p_\ell}$, for $0\leq \ell \leq m$, are in $G_p$.   
To this end, we consider two cases:   

{\it Case 1: $w_p\notin N_{\rm out}(P)$.} Since $G_p = (V_p, E_p)$ is strongly connected, there is a directed spanning tree $T_p = (V_p, E'_p)$, {\em without self-loop}, of $G_p$ rooted at the node~$\alpha_{p_0}$. 
We remove all the edges in $E_p - E'_p$ and obtain $G^{(p)}$. In other word, we replace  $G_p$ with $T_p$. 
It follows that the skeleton graph $S^{(p)}$ can be obtained from $S^{(p-1)}$ by replacing the node $w_p$ with $T_p$, the edge $w_iw_p$ with $w_i \alpha_{p_0}$, and the edges $w_pw_{j_\ell}$ with $\alpha_{p_\ell}w_{j_\ell}$ for $\ell = 1,\ldots, m$. It should be clear that $G^{(p)}$ satisfies the first two items of Definition~\ref{def:reduced}. We claim that the core of $S^{(p)}$ is still~$P$. To establish the claim, it suffices to show that the nodes of $T_p$ are successors of some node with a self-loop in $S^{(p)}$. We exhibit below such a node.  
Consider the (unique) path in $S^{(p-1)}$ from $\beta$ to $w_p$. Traversing the path, we let $w_{p'}$ be the first node such that $w_{p'}\notin P$. Then, $w_{p'}$ has a self-loop because otherwise $w_{p'}$ would belong to the core of $S^{(p-1)}$, contradicting the induction hypothesis that $P$ is the core.  
Since $w_p$ is not an out-neighbor of $P$, $w_{p'} \neq w_p$ and $w_{p'}$ belongs to $S^{(p)}$. It follows that all nodes of $T_p$ are successors of $w_{p'}$.

{\it Case 2: $w_p\in N_{\rm out}(P)$.} 
We again let $T_p = (V_p, E'_p)$ be a directed spanning tree of $G_p$, without self-loop and rooted at $\alpha_{p_0}$. 
Let $\alpha_{p'}$ be an in-neighbor of $\alpha_{p_0}$ in $G_p$; such a node exists since $G_p$ is  strongly connected and nontrivial. Let $G'_p:= T_p \cup \{\alpha_{p'}\alpha_{p_0}\}$. Then, $G'_p$ has a unique cycle $C'_p$ given by concatenating the path from $\alpha_{p_0}$ to $\alpha_{p'}$ in $T_p$ with the edge $\alpha_{p'}\alpha_{p_0}$. Let $E''_p$ be the edge set of $G'_p$. 
We obtain $G^{(p)}$ by removing the edges in $E_p - E''_p$ (i.e., by replacing $G_p$ with $G'_p$). The skeleton graph $S^{(p)}$ of $G^{(p)}$ can be described as follows: Let $S'_p$ be the skeleton graph $G'_p$ and $\pi'_p: G'_p \to S'_p$ be the graph homomorphism. Then, $S'_p$ is a directed tree and $w'_0:= \pi'_p(C'_p)$ is the root {\em with a self-loop}. Let $w'_\ell:= \pi'_p(\alpha_{p_{\ell}})$, for $\ell = 1,\ldots,m$ (these $(\ell + 1)$ nodes $w'_0,\ldots, w'_\ell$ may not be distinct). Then, $S^{(p)}$ can be obtained from $S^{(p-1)}$ by replacing the node $w_p$ with $S'_p$, the edge $w_iw_p$ with $w_i w'_{0}$, and the edges $w_pw_{j_\ell}$ with $w'_{\ell}w_{j_\ell}$ for $\ell = 1,\ldots, m$. It should be clear that $G^{(p)}$ satisfies the first two items of Definition~\ref{def:reduced}. Since the root $w'_0$ of $S'_p$ has a self-loop, every node of $S'_p$ is a successor of $w'_0$. Also, any successor $w_{p'}$ of $w_p$ in $S^{(p-1)}$, for $w'_p\neq w_p$, is now a successor of $w'_0$. Thus, $P$ remains as the core of $S^{(p)}$. Furthermore, since $w_p$ is an out-neighbor of $P$ in $S^{(p-1)}$ and since $S^{(p-1)}$ is a directed tree, $w_i$ must be a node of $P$. Since $w_iw'_0$ is an edge of $S^{(p)}$, $w'_0$ is an out-neighbor of~$P$. 

Finally, let $G':= G^{(q)}$, $S':= S^{(q)}$, and $\pi':= \pi^{(q)}$. By construction, $G'$ satisfies the third item of Definition~\ref{def:reduced}. Indeed, if a node $w'_p\in S'$ has a self-loop, then it comes from {\em Case 2} in the above edge-removal procedure, thus being an out-neighbor of $P$. The corresponding strong component $\pi'^{-1}(w'_p)$ is a cycle~$C'_p$.    
\end{proof}

\subsection{Walks and reachable node sets}\label{ssec:walksandreachablesets}
For the remainder of the paper, we will assume that $G\in \G_{n,1}$ is reduced. Each strong component of $G$ is either a cycle or a single node without self-loop. In this subsection, we will first characterize all the walks $\tau$ in $G$ with $\beta$ the starting node, and then describe sets of reachable nodes by walks of particular lengths.

\subsubsection{Characterization of walks in $G$}
Let $G^*$ be the core of $G$, which is a path without self-loops. It should be clear that if $\tau$ is a walk from $\beta$ to some $\alpha_i\in G^*$, then $\tau$ has to be a path (which is unique). We will now deal with the case where $\alpha_i\notin G^*$.   

Recall that $G_{\rm cyc}$ is the union of all {\em nontrivial} strong components of $G$. Since $G$ is reduced, $G_{\rm cyc}$ is a collection of disjoint cycles. We label these cycles as $G_p = (V_p, E_p)$, for $p = 1,\ldots, q$. 
Let $V_p^+$ be the successors of $V_p$ in $G$, and $G_p^+$ be the subgraph induced by $V_p^+$. 
Since $G$ is reduced, it follows from item~2 of Definition~\ref{def:reduced} that for each $G_p$, there is a unique node $\alpha_{p_0}$ in $G_p$ such that it is an out-neighbor of $G^*$:  
\begin{equation}\label{eq:definitialnode}
\{\alpha_{p_0}\} = N_{\rm out}(G^*) \cap G_p.
\end{equation}
For instance, in Figure~\ref{fig:ex-reduced}, $\alpha_2$ and $\alpha_3$ are the unique nodes connecting $G^*$ to the 3-cycle $\alpha_2\alpha_4\alpha_6\alpha_2$ and the 2-cycle $\alpha_3\alpha_5\alpha_3$, respectively. 

We make a few simple observations: 
First, note that there is a unique path in $G$, denoted by $\tau_{\beta\alpha_{p_0}}$, from $\beta$ to $\alpha_{p_0}$. All nodes of $\tau_{\beta\alpha_{p_0}}$, except $\alpha_{p_0}$, belong to $G^*$. Next, we express the cycle $G_p$ as  
$G_p = \alpha_{p_0}\alpha_{p_1}\cdots \alpha_{p_{\ell -1}} \alpha_{p_0}$. 
If we remove the edge $\alpha_{p_{\ell - 1}}\alpha_{p_0}$, then $G_p^+$ becomes a directed tree with $\alpha_{p_0}$ the root. This, in particular, implies that for any node $\alpha_i\in G_p^+$, there is a unique path, denoted by $\tau_{\alpha_{p_0}\alpha_i}$, from $\alpha_{p_0}$ to $\alpha_i$.  

For convenience, we denote by $G^m_p$ the closed walk in $G$ obtained by traversing $m$ times the cycle $G_p$, with $\alpha_{p_0}$ the starting (and ending) node. We allow $m$ to be $0$ and set $G_p^0 := \alpha_{p_0}$. 

The following result is an immediate consequence of the above arguments: 

\begin{Lemma}\label{lem:walks}
Let $\alpha_i\in G_p^+$ and $\tau$ be a walk in $G$ from $\beta$ to $\alpha_i$. Then, $\tau$ can be expressed as a concatenation of three walks: 
\begin{equation*}\label{eq:walkcase2}
\tau = \tau_{\beta \alpha_{p_0}} \, G_p^m \, \tau_{\alpha_{p_0}\alpha_i}.
\end{equation*}
\end{Lemma}

Note that if we set $m = 0$, then the walk $\tau$ in the above lemma becomes a path. It follows that for every $\alpha$-node $\alpha_i$ in $G$, there is a unique path from $\beta$ to $\alpha_i$. We denote this path by $\tau_{\beta\alpha_i}$ and define the {\bf depth} of $\alpha_i$ as
\begin{equation}\label{eq:defdepth}
\dep(\alpha_i):=\ell(\tau_{\beta\alpha_i}).
\end{equation} 
To illustrate, consider again the digraph in Figure~\ref{fig:ex-reduced}. The nodes of depth~1 are $\{\alpha_1,\alpha_2\}$, the nodes of depth~2 are $\{\alpha_3,\alpha_4\}$, $\{\alpha_5,\alpha_6\}$ depth~3, $\{\alpha_7,\alpha_8\}$ depth~4, and $ \alpha_9$ is of depth~5.

\subsubsection{On reachable sets}
For a positive~integer $j$, let $V(j)$ be the set of $\alpha$-nodes~$\alpha_i\in G$ such that there is a walk~$\tau$ of length~$j$ from~$\beta$ to~$\alpha_i$. For each $p = 1,\ldots, q$, we let
$$V_p^+(j) := V_p^+ \cap V(j).$$
We describe below these sets and present relevant properties, which will be used for computing the columns of the infinite matrix $C(A, b)$. 

For ease of notation, let
$\ell_p:= \ell(G_p)$. We start with the following result: 

\begin{Lemma}\label{lem:existuniquewalk}
For any positive integer~$j$,  
\begin{equation}\label{eq:existuniquewalk}
V_p^+(j) = \{ \alpha_i\in V_p^+ \mid \dep(\alpha_i) \leq j  \mbox{ and } \dep(\alpha_i) \equiv_{\ell_p} j\}. 
\end{equation}
Moreover, for any $\alpha_i \in V_p^+(j)$, there is a unique walk of length~$j$ from $\beta$ to $\alpha_i$. 
\end{Lemma} 

\begin{proof}
Consider all the walks $\tau$ from $\beta$ to $\alpha_i$. By Lemma~\ref{lem:walks}, 
$$\ell(\tau) = \ell(\tau_{\beta\alpha_i}) + m\ell_p = \dep(\alpha_i) + m\ell_p, \quad \mbox{for } m \in \N.$$
Thus, amongst these walks, there exists a $\tau$ of length~$j$ if and only if~\eqref{eq:existuniquewalk} holds. Moreover, such a walk is unique given by
$$
\tau = \tau_{\beta \alpha_{p_0}} \, G_p^{\frac{j-\dep(\alpha_i)}{\ell_p}} \, \tau_{\alpha_{p_0}\alpha_i}.
$$
This completes the proof.
\end{proof}

Let $\alpha_{p_0}$ be given as in~\eqref{eq:definitialnode}. Amongst all nodes in $G^+_p$, $\alpha_{p_0}$ has the minimum depth. Thus, $V_p^+(j) =  \varnothing$ if and only if $j < \dep(\alpha_{p_0})$. As~$j$ increases, the sequence of sets $V_p^+(j)$ will eventually be periodic.  We make it precise below. Let  
\begin{equation*}\label{eq:defdmindmax}
d_p:= \max_{\alpha_i\in V_p^+}\dep(\alpha_i).
\end{equation*}
Note that $d_p \geq \ell_p - 1$. 
We have the following result (an example illustrating the result will be given at the end of the subsection):

\begin{Lemma}\label{lem:decomVp}
If $j > d_p - \ell_p$, then 
\begin{equation}\label{eq:periodicV}
V_p^+(j + \ell_p)  =  V_p^+(j), 
\end{equation}
and, moreover,  
\begin{equation}\label{eq:partitionV}
V_p^+ = \sqcup_{k = 0}^{\ell_p - 1} V_p^+(j + k).
\end{equation}  
\end{Lemma}

\begin{proof}
We first establish~\eqref{eq:periodicV}. By~\eqref{eq:existuniquewalk}, $V^+_p(j) \subseteq V^+_p(j + \ell_p)$. Suppose to the contrary that there exists a node $\alpha_i$ such that 
$\alpha_i \in V^+_p(j + \ell_p) - V^+_p(j)$; 
then, using again~\eqref{eq:existuniquewalk}, we have that  
$j < \dep(\alpha_i) \leq j + \ell_p$ and $\dep(\alpha_i) \equiv_{\ell_p} j$. But then, $\dep(\alpha_i) = j + \ell_p > d_p$, which is a contradiction. 

We next establish~\eqref{eq:partitionV}. On one hand, each node $\alpha_i\in V^+_p$ belongs to  $V_p^+(j)$ for some~$j$ (precisely, for $j = \dep(\alpha_i) + m\ell_p$). On the other hand,  if $j \not\equiv_{\ell_p} j'$, then $V^+_p(j)$ does not intersect $V^+_p(j')$. Combining the arguments with~\eqref{eq:periodicV}, we conclude that~\eqref{eq:partitionV} holds. 
\end{proof}

We will now extend Lemma~\ref{lem:decomVp} to the subsets $V(j)$ for sufficiently large~$j$. To the end, let 
\begin{equation}\label{eq:defmaxlcm}
\ell_{\max} := \max_{p = 1}^q \{\ell_p\} \quad \mbox{and} \quad 
L:= \lcm_{p = 1}^q\{\ell_p\}.
\end{equation}
where $\lcm$ stands for the least common multiple. 

Recall that $V^+_{\rm cyc}$ is the set of successors of all nontrivial strong components  in $G$, and $G^+_{\rm cyc}$ is the subgraph induced by $V^+_{\rm cyc}$. Since $G$ is reduced,  
\begin{equation}\label{eq:partitionofGplus}
G^+_{\rm cyc} = \sqcup_{p = 1}^q G^+_p.
\end{equation}  
Also, recall that~$n^*$ is the order of $G^*$. We have the following result:

\begin{Lemma}\label{lem:repeatV}
Suppose that $j > n^*$ and $j > \max_{p = 1}^q (d_p - \ell_p)$;  
then, 
\begin{equation}\label{eq:partitionofVplus}
V(j + L)  =  V(j) \quad \mbox{and} \quad  V^+_{\rm cyc} = \cup_{k = 0}^{\ell_{\max} - 1} V(j + k). 
\end{equation}
Moreover, for any $k = 0,\ldots, \ell_{\max}-1$, 
\begin{equation}\label{eq:notcontain}
V(j+k) \subsetneq V^+_{\rm cyc} - V(j+k).
\end{equation} 
\end{Lemma}

\begin{proof}
Since $G^*$ is a path without self-loop, for any $j > n^*$, $V(j)$ is a subset of $V^+_{\rm cyc}$.  
The fact that~\eqref{eq:partitionofVplus} holds then follows directly from Lemma~\ref{lem:decomVp} and~\eqref{eq:partitionofGplus}. To establish~\eqref{eq:notcontain}, we let $G_p$ be such that $\ell_p = \ell_{\max}$. On one hand, we have $V^+_p(j + k) \subseteq V(j + k)$. On the other hand, it follows from~\eqref{eq:partitionV} and~\eqref{eq:partitionofGplus} that 
$$
V^+_p(j + k) \cap (V^+_{\rm cyc} - V(j+k)) = \varnothing.
$$
This completes the proof. 
\end{proof}

\begin{Example}\normalfont
Consider the reduced graph $G$ in Figure~\ref{fig:ex-reduced}. The core $G^*$ has $n^* = 2$ nodes. Let $G_1:= \alpha_3\alpha_5\alpha_3$ and $G_2:= \alpha_2\alpha_4\alpha_6\alpha_2$ be the two cycles in $G$. 
Then, $G_1^+$ is the subgraph induced by $V^+_1=\{\alpha_3,\alpha_5,\alpha_7,\alpha_9\}$ and $G_2^+$ is the subgraph induced by $V^+_2=\{\alpha_2,\alpha_4,\alpha_6, \alpha_8\}$.   
We have that $\ell_1 = 2$, $\ell_2 = 3$, $d_1 = 5$, and $d_2 = 4$, so $\ell_{\max} = 3$ and $L = \lcm\{\ell_1,\ell_2\}= 6$. For $j \geq 4$, we have that
$$
V_1^+(j) = 
\begin{cases} 
\{\alpha_3,\alpha_7\} & \mbox{if } j \equiv_2 0, \\
\{\alpha_5,\alpha_9\} & \mbox{if } j \equiv_2  1.
\end{cases}
$$
Similarly, for $j \geq 2$, we have that
$$
V_2^+(j) = 
\begin{cases} 
\{\alpha_6\} & \mbox{if } j \equiv_3 0, \\
\{\alpha_2,\alpha_8\} & \mbox{if } j \equiv_3  1, \\
\{\alpha_4\} & \mbox{if } j \equiv_3  2.
\end{cases}
$$
Finally, for $j \geq 4$, we obtain that
$$
V^+(j) = 
\begin{cases} 
\{\alpha_3,\alpha_6,\alpha_7\} & \mbox{if } j \equiv_6 0, \\
\{\alpha_2,\alpha_5,\alpha_8,\alpha_9\} & \mbox{if } j \equiv_6  1, \\
\{\alpha_3,\alpha_4,\alpha_7\} & \mbox{if } j \equiv_6  2, \\
\{\alpha_5,\alpha_6,\alpha_9\} & \mbox{if } j \equiv_6  3, \\
\{\alpha_2,\alpha_3,\alpha_7,\alpha_8\} & \mbox{if } j \equiv_6  4, \\
\{\alpha_4,\alpha_5,\alpha_9\} & \mbox{if } j \equiv_6  5.
\end{cases}
$$
Note that the $V^+(j)$'s satisfy~\eqref{eq:partitionofVplus} and~\eqref{eq:notcontain}. \hfill{\qed}
\end{Example}

\subsection{Construction of the $(A, b)$ pair}\label{ssec:constructionAb}
Let $G\in \G_{n,1}$ be reduced. In this subsection, we will construct a pair $(A, b)$ in $\V(G)$ and show, toward the end of the section, that the pair is averaged controllable.

Since the edges of $G = (V, E)$ correspond one-to-one to the ``free entries'' of the $(A, b)$ pair (specifically, edge $\alpha_j\alpha_i$ corresponds to entry $a_{ij}$ and edge $\beta\alpha_i$ corresponds to entry $b_i$), we construct the matrix pair by assigning to each edge of~$G$ a continuous function. 
As announced earlier, these functions will be of the form $f^{\nu(e)}$, where $f: \Sigma \to \R_{\geq 0}$ is a continuous function and $\nu: E\to \R_{\geq 0}$ determines the power of~$f$. We define~$f$ and~$\nu$ below.

\subsubsection{Construction of~$f$} 
Let $(\mathcal{U}, \phi)$ be a chart on $\Sigma$, i.e., $\mathcal{U}$ is an open set in $\Sigma$ and $\phi: \mathcal{U} \to \mathcal{V}$ is a diffeomorphism, with $\mathcal{V}\in \R^n$ an open neighborhood of~$0$ in $\R^n$.   
Without loss of generality, we can assume that $\mathcal{U}$ is contained in the support of the measure~$\mu$ (introduced at the beginning of Subsection~\ref{ssec:problemformulation}) and that
$\mathcal{V}$ contains the closed unit ball~$B$: 
$$B := \{z\in \R^n \mid \|z\| \leq 1\}.$$ 
Let $\tilde f: B \to \R$ be defined as: 
\begin{equation}\label{eq:deftildef}
\tilde f(z) := \|z\|.
\end{equation}
Let $K:= \phi^{-1}(B)$. 
We then define $f:\Sigma \to \R$ as 
\begin{equation}\label{eq:deff}
f(\sigma) := 
\begin{cases}
\tilde f(\phi(\sigma)) & \mbox{if } \sigma\in K, \\
1 & \mbox{otherwise}.
\end{cases}
\end{equation}
It should be clear that $f$ is continuous and that the set $\{f^k\}_{k\in \N}$ is linearly independent, where $f^0 = \mathbf{1}$ is the constant function taking value $1$ everywhere.   
Further, we let $\mathbb{F}$ be the uniform closure of the space spanned by~$f^k$ for all $k\in \N$, 
$$\mathbb{F}:= \mbox{uniform closure of the span}\{f^k\}_{k\in \N}.$$ 
The following lemma follows directly from the Stone-Weierstrass theorem:

\begin{Lemma}\label{lem:spanf}
The space $\bF$ comprises all continuous functions $g:\Sigma \to \R$ such that $g$ is constant over $\Sigma - K$ and that $g(\sigma) = g(\sigma')$ for any two points $\sigma, \sigma'\in K$, with $\phi(\sigma) = \phi(\sigma')$.    
\end{Lemma}

\subsubsection{Construction of $\nu$}  
Recall that $\dep(\alpha_i)$ is the depth of $\alpha_i$ defined in~\eqref{eq:defdepth}. Relabel, if necessary, the $\alpha$-nodes in~$G$ such that 
\begin{equation*}\label{eq:relable}
\mbox{if } \dep(\alpha_i) < \dep(\alpha_j), \mbox{ then } i < j. 
\end{equation*} 
This can be done by using, for example, the breadth-first search algorithm. The labeling of the $\alpha$-nodes of the graph shown in Figure~\ref{fig:ex-reduced} satisfies this property.

We will now partition the edge set $E$ into five subsets $E_i$, for $i = 1,\ldots,5$, and define $\nu |_{E_i}$. Let $V_{\rm appx}:= V^+_{\rm cyc} - V_{\rm cyc}$, and $G_{\rm appx} = (V_{\rm appx}, E_{\rm appx})$ be the subgraph induced by $V_{\rm appx}$. We call $G_{\rm appx}$ the appendix of $G$. Define the subsets $E_i$ of $E$ as follows:
\begin{equation}\label{eq:5subsets}
\begin{cases}
E_1:= E^*, \\
E_2:= \{v_i \alpha_j \in E \mid v_i \in G^* \mbox{ and } \alpha_j \in G_{\rm cyc}\}, \\
E_3:= E_{\rm cyc}, \\
E_4:= \{\alpha_i\alpha_j \in E \mid \alpha_i \in G_{\rm cyc} \mbox{ and } \alpha_j \in G_{\rm appx}\}, \\
E_5:= E_{\rm appx}.
\end{cases}
\end{equation}
The edges in $E_2$ link nodes from $G^*$ to $G_{\rm cyc}$, and edges in $E_4$ link nodes from $G_{\rm cyc}$ to $G_{\rm appx}$.  

Let $\lambda$ be a positive and {\em irrational} number, which we fix in the sequel. We now define $\nu(e)$ as follows:   

{\it Case 1: $e\in E_1$.} We set $\nu(e):=0$.

{\it Case 2: $e\in E_2$.} Let $G_p$ be the cycle incident to~$e$, and $\alpha_{p_0}$ be given as in~\eqref{eq:definitialnode}. Then, $e$ can be written as $e = v_i\alpha_{p_0}$.   
We set $\nu(e):= p_0 \lambda$. 

{\it Case 3: $e \in E_3$.} Let $G_p$ be the cycle that contains~$e$. We write $G_p$ explicitly as $G_p = \alpha_{p_0} \alpha_{p_1} \cdots \alpha_{p_{\ell_p-1}}\alpha_{p_0}$. Then, we set
$$
\nu(e):= 
\begin{cases}
\frac{\ell_p}{L} & \mbox{if } e = \alpha_{p_{\ell_p - 1}} \alpha_{p_0} \\
0 & \mbox{otherwise}, 
\end{cases}
$$
where $L$ is defined in~\eqref{eq:defmaxlcm}. 

{\it Case 4: $e\in E_4$.} Let $G_p$ be the cycle and $\alpha_j\in G_{\rm appx}$ be the node incident to~$e$. We still let $\alpha_{p_0}$ be given as in~\eqref{eq:definitialnode}. We then set $\nu(e) = (j - p_0)\lambda$. Note that $(j - p_0) > 0$ because $\dep(\alpha_j) > \dep(\alpha_{p_0})$. 

{\it Case 5: $e\in E_5$.} We write $e = \alpha_i \alpha_j$ and set $\nu(e) := (j - i) \lambda$. Note that $(j - i) > 0$ because $\dep(\alpha_j) - \dep(\alpha_i) = 1$. 

We illustrate the map $\nu$ in Figure~\ref{fig:ex-reducedweight} for the graph shown in Figure~\ref{fig:ex-reduced}. 

\begin{figure*}[htb!]
\begin{center}
\begin{tikzpicture}[scale=1.0,node distance=1.4cm]   
\node[main node] (1)  {$\alpha_1$};  
\node[main node] (b) [below of= 1] {$\beta$};
    \node[main node] (2) [left of=1]  {$\alpha_2$};    
    \node[main node] (5) [above 
    left  of=2]  {$\alpha_4$};
    \node[main node] (6) [below left of= 5] {$\alpha_6$};    
    \node[main node] (7) [left of= 6]  {$\alpha_8$};

    \node[main node] (3) [right of= 1] {$\alpha_3$};
    \node[main node] (4) [above right of= 3] {$\alpha_5$};    
    \node[main node] (8) [below right of= 4] {$\alpha_7$};
    \node[main node] (9) [right of= 8] {$\alpha_9$};    
    
    \path[draw,shorten >=2pt,shorten <=2pt]    
    (1) edge[-latex] node[above] {\footnotesize $3\lambda$} (3)    
    (6) edge[-latex] node[below] {\footnotesize $\frac{1}{2}$} (2)
    (6) edge[-latex] node[above] {\footnotesize $6\lambda$} (7)
    (4) edge[-latex] node[label={[font=\footnotesize,xshift=.2cm,yshift=-.2cm]:$4\lambda$}] {} (8)
    (5) edge[-latex] node[label={[font=\footnotesize,xshift=-.1cm,yshift=-.2cm]:0}] {} (6)  
    (2) edge[-latex] node[label={[font=\footnotesize,xshift=.1cm,yshift=-.2cm]:0}] {} (5)
    (8) edge[-latex] node[above] {\footnotesize $2\lambda$} (9)
    (b) edge[-latex] node[right] {\footnotesize $0$}  (1)
    (b) edge[-latex] node[left] {\footnotesize $2\lambda$} (2)    
    (3) edge[bend left=15,-latex] node[label={[font=\footnotesize,xshift=-.1cm,yshift=-.2cm]:0}] {} (4)    
    (4) edge[bend left=15,-latex] node[label={[font=\footnotesize,xshift=.2cm,yshift=-0.55cm]:$\frac{1}{3}$}] {}
    (3);              
    \end{tikzpicture}
\end{center}
\caption{Illustration of the map $\nu$ on the digraph $G$ shown in Figure~\ref{fig:ex-reduced}. The values of $\nu(e)$ are treated as edge weights and are shown on the corresponding edges $e$. The 5 subsets $E_i$, introduced in~\eqref{eq:5subsets}, are given by $E_1:= \{\beta\alpha_1\}$, $E_2=\{\beta\alpha_2,\alpha_1\alpha_3\}$, $E_3=\{\alpha_3\alpha_5,\alpha_5\alpha_3\}\cup \{\alpha_2\alpha_4,\alpha_4\alpha_6,\alpha_6\alpha_2\}$, $E_4 = \{\alpha_5\alpha_7,\alpha_6\alpha_8\}$, and $E_5 = \{\alpha_7\alpha_9\}$. Note that $ L = 6$.} \label{fig:ex-reducedweight}
\end{figure*}
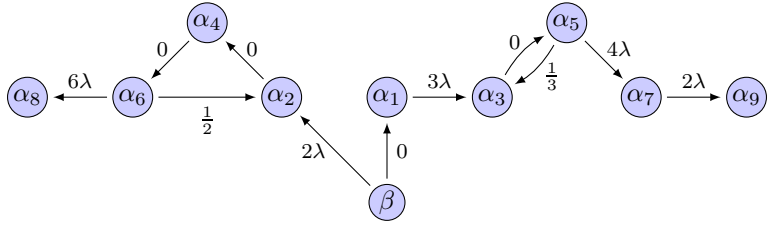

\subsection{Computation of $C(A, b)$}\label{ssec:computationcab} 
Let $(A, b)\in \V(G)$ be given as above, and $C(A, b)$ be defined as in~\eqref{eq:defcab}. 
In this subsection, we present explicit expressions of the entries of the matrix $C(A, b)$.      

To this end, let $\cT$ be the set of walks in $G$. We extend $\nu: E\to \R_{\geq 0}$ by defining the map over the set~$\cT$.  
Specifically, for a walk $\tau = v_{i_1} \cdots v_{i_k}$, we define
\begin{equation}\label{eq:deftauonwalks}
\nu(\tau) := \sum_{j = 1}^{k-1} \nu(v_{i_j}v_{i_{j+1}}).
\end{equation}

Let $\tau$ be a walk from $\beta$ to some $\alpha_i$. We compute below $\nu(\tau)$. 
The case where $\alpha_i \in G^*$ is simple: The only walk from $\beta$ to $\alpha_i$ in $G$ is the path $\tau_{\beta\alpha_i}$, which is contained in $G^*$. Since $\nu(e) = 0$ for all $e\in G^*$, $\nu(\tau_{\beta\alpha_i}) = 0$. 
For the other cases, we have the following result: 

\begin{Lemma}\label{lem:degreesforCAB}
Let $\alpha_i \in G_p^+$ and $\tau$ be a walk from $\beta$ to $\alpha_i$. Then,  
\begin{equation*}\label{eq:degreesforCAB}
\nu(\tau)= 
\begin{cases}
p_0 \lambda + \frac{\ell(\tau)-\dep(\alpha_i)}{L} & \mbox{if } \alpha_i \in G_p, \vspace{.1cm} \\
i \lambda   + \frac{\ell(\tau)-\dep(\alpha_i)}{L} & \mbox{if } \alpha_i \in G_p^+ - G_p. 
\end{cases}
\end{equation*}
\end{Lemma}

\begin{proof}
By Lemma~\ref{lem:walks}, we can express $\tau$ as
$\tau = \tau_{\beta\alpha_{p_0}} G_p^m \tau_{\alpha_{p_0}\alpha_i}$. 
It follows from~\eqref{eq:deftauonwalks} that 
$$\nu(\tau) = \nu(\tau_{\beta\alpha_{p_0}}) + m \nu(G_p) + \nu(\tau_{\alpha_{p_0}\alpha_i}).$$
By construction of~$\nu$, we have that $\nu(\tau_{\beta\alpha_{p_0}}) = p_0 \lambda$, $\nu(G_p) = \frac{\ell_p}{L}$, and 
$$
\nu(\tau_{\alpha_{p_0}\alpha_i}) = 
\begin{cases}
0 & \mbox{if } \alpha_ i \in G_p,\\
(i - p_0)\lambda & \mbox{if } \alpha_i \in G^+_p - G_p.
\end{cases}
$$
Finally, by the fact that $\tau_{\beta\alpha_i} = \tau_{\beta\alpha_{p_0}}\tau_{\alpha_{p_0}\alpha_i}$ and $\dep(\alpha_i) = \ell(\tau_{\beta\alpha_i})$, we have 
$$
m = \frac{\ell(\tau)-\dep(\alpha_i)}{\ell_p},
$$
which completes the proof. 
\end{proof}

For a later purpose, we express $\nu(\tau)$ as a sum of two parts: 
$$\nu(\tau) = \nu_1(\tau) \lambda  +  \nu_2(\tau)\frac{1}{L}, $$
where $\nu_1(\tau)$ and $\nu_2(\tau)$ are given by
\begin{equation}\label{eq:defnu1}
\nu_1(\tau) := 
\begin{cases}
0 &  \mbox{if } \alpha_i\in G^*, \\
p_0 & \mbox{if } \alpha_i\in G_p \mbox{ for some cycle }G_p, \\
i & \mbox{if } \alpha_i \in G_{\rm appx}.
\end{cases} 
\end{equation}
and
\begin{equation}\label{eq:defnu2}
\nu_2(\tau) := 
\begin{cases}
0 & \mbox{if } \alpha_i\in G^*,  \\ 
\ell(\tau)-\dep(\alpha_i) & \mbox{if } \alpha_i \in G_{\rm cyc}^+.
\end{cases}
\end{equation}
We call $\nu_1(\tau)$ the {\bf irrational part}  and $\nu_2(\tau)$ the {\bf rational part} of $\nu(\tau)$.

Recall that $V(j)$ comprises all the $\alpha_i$ such that there is a walk $\tau$ of length~$j$ in $G$ from $\beta$ to $\alpha_i$. By Lemma~\ref{lem:existuniquewalk}, such a walk $\tau$ is unique for each $\alpha_i \in V(j)$.  Also, recall that the function $f:\Sigma\to \R_{\geq 0}$ is defined in~\eqref{eq:deftildef} and~\eqref{eq:deff}.   

We  have the following result for the matrix $C(A,b)$ introduced in~\eqref{eq:defcab}: 

\begin{Lemma}\label{lem:computentries}
Let $c_{ij}$ be the $ij$th entry of $C(A, b)$. Then, 
$$
c_{ij} = 
\begin{cases}
\displaystyle\int_\Sigma f^{\nu(\tau)} \rd\mu & \mbox{if } \alpha_i\in V(j), \\
0 & \mbox{otherwise}, 
\end{cases}
$$
where $\tau$ is the unique walk of length~$j$ from $\beta$ to $\alpha_i$.
\end{Lemma}

\begin{proof}
The nonzero entries of $(A, b)$ are by construction $f^{\nu(e)}$, with $e\in G$ the corresponding edges. It follows that the $i$th entry of the vector $A^{j-1}b$ is given by 
$\sum_{\tau} f^{\nu(\tau)}$, where the sum is over all walks in $G$ of length~$j$ from $\beta$ to~$\alpha_i$. By Lemma~\ref{lem:existuniquewalk}, such a walk, if exists, is unique. This completes the proof.  
\end{proof}

\subsection{On a class of submatrices of $C(A, b)$}\label{ssec:submatricescab} 
In this subsection, we focus on a special class of submatrices of $C(A, b)$, and show that every such submatrix has full rank. The result will be used in the next subsection to establish that the matrix $C(A, b)$ itself has full rank. 

Let $C_j$ be the $j$th column of $C(A, b)$, so $c_{ij}$ is the $i$th entry of~$C_j$. We define the support of $C_j$ as 
$\supp (C_j) := \{\alpha_i \in V_\alpha \mid c_{ij} \neq 0\}$.  
By construction, $f$ is positive almost everywhere. It follows from Lemma~\ref{lem:computentries} that 
\begin{equation}\label{eq:suppequalvj}
\supp (C_j) = V(j) \quad \mbox{for all } j \geq 1. 
\end{equation}

Let $j^*$ be a positive integer such that
\begin{equation}\label{eq:jstar}
j^* > n^* \quad \mbox{and} \quad j^* > \max_{p = 1}^q (d_p - \ell_p),
\end{equation} 
so $j^*$ satisfies the hypothesis of Lemma~\ref{lem:repeatV}.  
Let $L$ be given as in~\eqref{eq:defmaxlcm}. 
Consider the column vectors $C_{j^* + kL}$, for $k\in \N$. By Lemma~\ref{lem:repeatV} and~\eqref{eq:suppequalvj}, all these columns share the same support. 
Let $V' = \{\alpha'_{1},\ldots, \alpha'_{m}\}$ be the common support, which is a subset of $V^+_{\rm cyc}$. 
Let $C'$ be the submatrix of $C$ obtained by first taking the columns $C_{j^* + kL}$, for all $k\in \N$, and then removing the zero rows, i.e.,   
\begin{equation}\label{eq:submatrixC'}
C':= \left [ C_{j^*}, C_{j^* + L}, C_{j^* + 2L}, \cdots  \right ] \big |_{V'}.
\end{equation}

Let $c'_{ij}$ be the $ij$th entry of $C'$. We can express $c'_{ij}$ explicitly. 
Since $\alpha'_j \in V'$, there exists a unique walk, denoted by $\tau_{ij}$, of length $(j^* + (i-1)L)$ from $\beta$ to $\alpha'_j$. Also, note that $\tau_{ij}$ can be obtained from $\tau_{1j}$ by inserting the closed walk $G_p^{(i - 1)L/\ell_p}$ for some $p = 1,\ldots, q$. Thus, by Lemma~\ref{lem:degreesforCAB}, 
$$
\nu(\tau_{ij}) = \nu(\tau_{1j}) + \frac{(i-1)L}{\ell_p} \nu(G_p) = \nu(\tau_{1j}) + i - 1.
$$
For ease of notation, let $\tau_j:= \tau_{1j}$. 
It then follows from Lemma~\ref{lem:computentries} that
\begin{equation}\label{eq:entriesofC'}
c'_{ij} = \int_\Sigma f^{\nu(\tau_j) + i-1} \rd \mu.  
\end{equation}

We will now establish the following result: 

\begin{Lemma}\label{lem:fullranksubmatrix}
The submatrix $C'$ given in~\eqref{eq:submatrixC'} has full rank (i.e., rank $m$). 
\end{Lemma}

\begin{proof}
The proof has two parts. In the first part, we show that $\nu(\tau_1),\ldots,\nu(\tau_m)$ are pairwise distinct. Building upon this fact and Lemma~\ref{lem:spanf}, we show in the second part that the matrix $C'$ has full rank. 

\xc{Part 1: Proof that $\nu(\tau_1),\ldots,\nu(\tau_m)$ are pairwise distinct.} 
For any two numbers $\nu(\tau_j)$ and $\nu(\tau_{j'})$, with $j\neq j'$, we show that their difference is nonzero. 
From Lemma~\ref{lem:degreesforCAB}, we have that
\begin{equation}\label{eq:diffnu}
\nu(\tau_j) - \nu(\tau_{j'}) = \\ (\nu_1(\tau_j) - \nu_1(\tau_{j'}))\lambda + (\nu_2(\tau_j) - \nu_2(\tau_{j'}))\frac{1}{L},
\end{equation}
where $\nu_1$ and $\nu_2$ are defined in~\eqref{eq:defnu1} and~\eqref{eq:defnu2}, respectively. 
Since $\lambda$ is irrational, to show that~\eqref{eq:diffnu} is nonzero, it suffices to show that $\nu_1(\tau_j) \neq \nu_1(\tau_{j'})$.

The ending nodes  of $\tau_j$ and $\tau_{j'}$ are $\alpha'_{j}$ and $\alpha'_{j'}$, respectively. 
We claim that $\alpha'_{j}$ and $\alpha'_{j'}$ do not belong to the same cycle. 
Suppose not, say $\alpha'_j, \alpha'_{j'}\in G_p$ for some $p = 1,\ldots, q$; 
then, by Lemma~\ref{lem:walks} and by the fact that $\tau_j$ and $\tau_{j'}$ have the same length~$j^*$,  
the two walks $\tau_j$ and $\tau_{j'}$ have to be the same:
$$
\tau_j = \tau_{j'} = \tau_{\beta\alpha_{p_0}} G_p^m \tau_{\alpha_{p_0}\alpha_{p_k}},
$$
where $\alpha_{p_k}$ is the unique node in $G_p$ such that $k \equiv_{\ell_p} j^*- \dep(\alpha_{p_0})$. But then, $\alpha'_{j} = \alpha'_{j'} = \alpha_{p_k}$, which is a contradiction. Since $\alpha'_{j}$ and $\alpha'_{j'}$ do not belong to the same cycle, we conclude from~\eqref{eq:defnu1} that $\nu_1(\tau_j) \neq \nu_1(\tau_{j'})$. 

\xc{Part 2: Proof that $C'$ has full rank.} We show that the only solution $v\in \R^m$ to $v^\top C'$ is~$v = 0$. We write $v  = (v_1,\ldots, v_m)$ and let $g : = \sum_{j = 1}^m v_j f^{\nu(\tau_j)}$. Then, by~\eqref{eq:entriesofC'}, $v^\top M = 0$ implies that 
\begin{equation}\label{eq:riesz}
\int_\Sigma g f^k \rd \mu = 0, \quad \mbox{for all } k \in \N.  
\end{equation}
By Lemma~\ref{lem:spanf}, any function $f^r$, for $r\in \R_{\geq 0}$, belongs to $\bF$. Thus, $g$ belongs to $\bF$ as well. 
Since the uniform closure of the linear span of $\{f^k\}_{k\in \N}$ is $\bF$,~\eqref{eq:riesz} can hold if and only if $g = 0$. 
Because $\nu(\tau_1), \ldots, \nu(\tau_m)$ are pairwise distinct, the functions 
$f^{\nu(\tau_1)}, \ldots, f^{\nu(\tau_m)}$ are linearly independent. Thus, $\sum_{j = 1}^m v_jf^{\nu(\tau_j)} = 0$ can hold if and only if $v = 0$. 
\end{proof}

\subsection{Proof that $(A, b)$ is averaged controllable}\label{ssec:cabavecontrol}
In this subsection, we complete the proof for the sufficiency part of Theorem~\ref{thm:main}. It remains to establish the following result:

\begin{Proposition}\label{prop:sufficiency}
    The matrix $C(A, b)$ has full rank.
\end{Proposition}

\begin{proof}
We exhibit below $n$ linearly independent column vectors out of $C(A, b)$.

Recall that $n^*$ is the order of the core $G^*$. 
We claim that the first $n^*$ columns of $C(A, b)$ are linearly independent. 
    To wit, note that $G^*$ is a path without self-loop. 
    For any node $\alpha_i\in G^*$, the $i$th row has a unique {\em nonzero} entry, namely, $c_{i,\dep(\alpha_i)}$. More specifically, we have that
    \begin{equation*}\label{eq:rowforGplus}
    c_{i j} = \delta_{\dep(\alpha_i), j} \int_\Sigma f^{\nu(\tau_{\beta\alpha_i})} \rd \mu  \\ = \delta_{\dep(\alpha_i), j} \int_{\Sigma} \mathbf{1} \rd \mu = \delta_{\dep(\alpha_i), j}, 
    \end{equation*}
    where $\delta$ is the Kronecker delta. This establishes the claim. If $n^* = n$, then the proof is complete. Otherwise, 
    we let $n^+ :=  n - n^*$ be the order of $G^+_{\rm cyc}$.  
    Let $j^*$ be a positive integer satisfying~\eqref{eq:jstar}.  By Lemma~\ref{lem:repeatV}, $V(j)$ does not intersect $V^*$ for all $j \geq j^*$.  
    Thus, if there exist $n^+$ linearly independent vectors out of $\{C_j\}_{j \geq j^*}$, then these $n^+$ vectors, together with the first $n^*$ columns of $C(A, b)$, form a basis of $\R^n$. We exhibit below these $n^+$ columns.

    Let $\ell_{\max}$ and $L$ be given as in~\eqref{eq:defmaxlcm}.   
    We define subsets $V'_\ell$, for $\ell = 0,\ldots, \ell_{\max} -1$, of $V^+_{\rm cyc}$ as follows: Set $V'_0 := V(j^*)$ and 
    $$
    V'_\ell := V(j^* + \ell) - \cup_{k = 0}^{\ell-1} V(j^* + k), \quad \mbox{for } 1\leq \ell \leq  \ell_{\max}-1. 
    $$
    By Lemma~\ref{lem:repeatV}, these subsets are nonempty and form a partition of $V^+_{\rm cyc}$.  
    Let $m_\ell := |V'_\ell|$, so $n^+ = \sum_{\ell = 0}^{\ell_{\max}-1} m_\ell$. 
    Next, for each $\ell = 0,\ldots, \ell_{\max}-1$, we define subsets of $\R^n$ as 
    $$\mathbf{C}_\ell:=\{C_{j^* + \ell + kL}\mid k\in \N\}.$$  
    Since $\ell_{\max}\leq L$, if $\ell \neq \ell'$, then $\mathbf{C}_\ell \cap \mathbf{C}_{\ell'} = \varnothing$. We will now select for each $\ell = 0,\ldots, \ell_{\max} - 1$, $m_\ell$ columns out of $\mathbf{C}_\ell$, and show that the total $n^+$ columns are linearly independent. 
    
    By Lemma~\ref{lem:repeatV} and~\eqref{eq:suppequalvj}, all the columns in $\mathbf{C}_\ell$ share the same support, given by $V(j^* + \ell)$, which contains $V'_\ell$ as a subset. 
    Since the subsets $V'_\ell$ form a partition of $V^+_{\rm cyc}$, 
    it suffices to show that for each $\ell$, there exist $m_\ell$ vectors $C_{j_1},\ldots, C_{j_{m_\ell}}$ out of $\mathbf{C}_\ell$ such that the sub-vectors
    $
    C_{j_1} |_{V'_\ell},\ldots,C_{j_{m_\ell}} |_{V'_\ell}
    $
    are linearly independent. 
    But this follows from Lemma~\ref{lem:fullranksubmatrix}: To wit, the following submatrix of $C(A, b)$: 
    \begin{equation}\label{eq:submatrixl}
    \left [C_{j^*+\ell}, C_{j^* + \ell + L}, C_{j^* + \ell + 2L}, \cdots \right ] \big |_{V(j^* +\ell)} 
    \end{equation}
    has full rank, so its rows are linearly independent. It follows that if replace $V(j^* + \ell)$ with its subset $V'_\ell$, then the resulting submatrix is still of full rank, i.e., of rank~$m_\ell$. This completes the proof.   
\end{proof}    

\section{Conclusions}
We have provided a complete characterization of sparsity patterns $G$ that can sustain averaged controllability within the class of linear ensemble systems with single inputs. Specifically, we have shown in Theorem~\ref{thm:main} that it is necessary and sufficient that the core $G^*$ (introduced in Definition~\ref{def:core})  contains a directed spanning path for $G$ to be structurally averaged controllable. We will extend the result to the multi-input case in future work.

\printbibliography


\end{document}